\documentclass[leqno, 12pt]{amsart}
\usepackage{epic, eepic, amsfonts, latexsym,amsmath, amssymb, graphicx,multicol,mathrsfs,color,xypic, amscd, amsxtra, verbatim, paralist, xspace, url, euscript, stmaryrd, newcent}
\usepackage{hyperref}

\newenvironment{enumeratea} {\begin{enumerate}[\upshape (a)]} {\end{enumerate}}
\newenvironment{enumeratei} {\begin{enumerate}[\upshape (i)]} {\end{enumerate}}
\newenvironment{enumerate1} {\begin{enumerate}[\upshape (1)]} {\end{enumerate}}

\newtheorem{lem}{Lemma}[section]
\newtheorem{thm}[lem]{Theorem}

\newtheorem{conj}[lem]{Conjecture}
\newtheorem{prop}[lem]{Proposition}
\newtheorem{cor}[lem]{Corollary}

\theoremstyle{definition}
\newtheorem{defn}[lem]{Definition}
\newtheorem{example}[lem]{Example}

\theoremstyle{remark}
\newtheorem{remark}[lem]{Remark}

\numberwithin{equation}{section}

\newcommand{\epf}{\qed \vspace{+10pt}}

   \newcommand\cF{\mathcal{F}} \newcommand\cG{\mathcal{G}}\newcommand\cI{\mathcal{I}}\newcommand\cJ{\mathcal{J}}\newcommand\cK{\mathcal{K}}\renewcommand\cL{\mathcal{L}}\newcommand\cO{\mathcal{O}}\newcommand\cU{\mathcal{U}}\newcommand\cV{\mathcal{V}}\newcommand\cW{\mathcal{W}}\newcommand\cX{\mathcal{X}}\newcommand\cY{\mathcal{Y}}\newcommand\cZ{\mathcal{Z}}

\renewcommand\AA{\mathbb{A}}\newcommand\GG{\mathbb{G}}\newcommand\PP{\mathbb{P}}
\newcommand\ZZ{\mathbb{Z}}

   \newcommand\fF{\mathfrak{F}} \newcommand\fG{\mathfrak{G}}\newcommand\fI{\mathfrak{I}}\newcommand\fJ{\mathfrak{J}}\newcommand\fR{\mathfrak{R}}\newcommand\fU{\mathfrak{U}}\newcommand\fW{\mathfrak{W}}\newcommand\fX{\mathfrak{X}}\newcommand\fY{\mathfrak{Y}}\newcommand\fZ{\mathfrak{Z}}

\newcommand\fc{\mathfrak{c}} \newcommand\fs{\mathfrak{s}}   \newcommand\ft{\mathfrak{t}}  \newcommand\fe{\mathfrak{e}} \newcommand\fid{\mathfrak{i}} \newcommand\fq{\mathfrak{q}}

\newcommand\id{\mathrm{id}}

\newcommand\arr{\ifinner\to\else\longrightarrow\fi}

\newcommand\mapsonto{\twoheadrightarrow}
\newcommand\hookarr{\hookrightarrow}

\newcommand\im{\operatorname{im}}
\newcommand{\QCoh}{\operatorname{QCoh}}

\renewcommand{\setminus}{\smallsetminus}

\newcommand{\red}{_{\mathrm{red}}}

\newcommand{\h}{\widehat}

\renewcommand{\ss}{\operatorname{ss}}
\newcommand{\s}{\operatorname{s}}
\newcommand{\rrarrows}{\rightrightarrows}

\newcommand{\Eq}{\operatorname{Eq}}

\newcommand{\sSpec}{\operatorname{\mathcal{S}{\it pec}}}

\newcommand{\Aut}{\operatorname{Aut}}

\newcommand{\oh}{\cO}

\newcommand{\Spec}{\operatorname{Spec}}

\newcommand{\tensor} {\otimes}

\renewcommand{\tilde}{\widetilde}

\newcommand{\Spf}{\operatorname{Spf}}

\newcommand{\iso}{\stackrel{\sim}{\arr}}

\newcommand{\ilim}{{\displaystyle \lim_{\longleftarrow}}\,}
\newcommand{\dlim}{{\displaystyle \lim_{\longrightarrow}}\,}
\newcommand{\htensor}{\h{\tensor}}

\renewcommand{\hat}{\widehat}

\title{Local properties of good moduli spaces}
\author{\textsc{Jarod Alper}}

\begin{document}
\begin{abstract}
We study the local properties of Artin stacks and their good moduli spaces, if they exist.  We show that near closed points with linearly reductive stabilizer, Artin stacks formally locally admit good moduli spaces.  In particular, the geometric invariant theory is developed for actions of linearly reductive group schemes on formal affine schemes.  We also give conditions for when the existence of good moduli spaces can be deduced from the existence of \'etale charts admitting good moduli spaces.
\end{abstract}

\maketitle
\footnote[0]
	{2000 \textit{Mathematics Subject Classification}.
	Primary 14D23; Secondary 13A50.}
\footnote[0]{\textit{Keywords and phrases. Artin stacks, good moduli spaces, invariant theory}}

\section{Introduction}

We address the question of whether good moduli spaces for an Artin stack can be constructed ``locally."  The main results of this paper are:  (1) good moduli spaces exist formally locally around points with linearly reductive stabilizer and (2) sufficient conditions are given for the Zariski-local existence of good moduli spaces given the \'etale-local existence of good moduli spaces.  We envision that these results may be of use to construct moduli schemes of Artin stacks without the classical use of geometric invariant theory and semi-stability computations.

The notion of a \emph{good moduli space} was introduced in \cite{alper_good_arxiv} to associate a scheme or algebraic space to Artin stacks with nice geometric properties reminiscent of Mumford's good GIT quotients.  While good moduli spaces cannot be expected to distinguish between all points of the stack, they do parameterize points up to orbit closure equivalence.  See Section \ref{notation_sec} for the precise definition of a good moduli space and for a summary of its properties.

While the paper \cite{alper_good_arxiv} systematically develops the properties of good moduli spaces, the existence was only proved in certain cases.  For instance, if $\cX = [\Spec A / G]$ is a quotient stack of an affine by a linearly reductive group, then $\cX \arr \Spec A^G$ is a good moduli space (\cite[Theorem 13.2]{alper_good_arxiv}).  Additionally, for any quasi-compact Artin stack $\cX$ with a line bundle $\cL$, there is a naturally defined semi-stable locus $\cX^{\ss}_{\cL}$ and stable locus $\cX^{\s}_{\cL}$ such that $\phi: \cX^{\ss}_{\cL} \arr Y$ is a good moduli space where $Y$ is a quasi-projective scheme, and there is an open subscheme $V \subseteq Y$ such that $\phi^{-1}(V) = \cX^{\s}_{\cL}$ and $\phi|_{\cX^{\s}_{\cL}}$ is a coarse moduli space (\cite [Theorem 11.14]{alper_good_arxiv}).

One might dream that there is some topological criterion guaranteeing existence of a good moduli space in the same spirit of the finite inertia hypothesis guaranteeing the existence of a coarse moduli space. One might pursue the following approach:
\begin{enumerate1}
\item Show that good moduli spaces exist locally around closed points.
\item Show that these patches glue to form a global good moduli space.
\end{enumerate1}

We are tempted to conjecture that if $x \in |\cX|$ is a closed point of an Artin stack with linearly reductive stabilizer, then there exists an open substack $\cU \subseteq \cX$ containing $x$ such that $\cU$ admits a good moduli space.  However, Example \ref{example} shows that this is too much to hope for, and it is unclear what the additional requirement should be to guarantee local existence of a good moduli space.

While we cannot establish the existence of good moduli spaces Zariski-locally or \'etale-locally, we show that formally locally good moduli spaces exist around closed points $\xi \in |\cX|$ with linearly reductive stabilizer.  Denote by $\cX_i$ the nilpotent thickenings of the induced closed immersion $\cG_{\xi} \hookarr \cX$. Section \ref{formal_good_sec} is devoted to making precise the following statement:  if $\hat{\cX}$ is the ``completion of $\cX$ at $\xi$'', then $\hat{\cX} \arr \Spf \ilim \Gamma(\cX_i, \oh_{\cX_i})$ is a good moduli space. 

We prove in Section \ref{formal_good_sec} that if there exists a good moduli space, then this formally local description is correct.  Precisely, we prove the following:
\begin{thm} \label{theorem-formal-local}
Suppose $\cX$ is an Artin stack of finite type over $\Spec k$ where $k$ is a field and $\phi: \cX \arr Y$ is a good moduli space.  Let $x: \Spec k \arr \cX$ be a closed point with image $y = \phi(x)$.  Let $\cX_i$ be the nilpotent thickenings of the induced closed immersion $BG_x \hookarr \cX$.  There are isomorphisms $\cX_i \cong [\Spec A_i / G_x]$ which induces an action of $G_x$ on $\Spf A$ where $A=\ilim A_i$.  There are isomorphisms of topological rings
$$\xymatrix{
\hat{\oh}_{Y,y} \ar[r] \ar[rd]		&  \ilim (A_i^{G_x}) \ar[d]\\
							& A^{G_x} .
}$$
\end{thm}

In particular, the formal local ring $\hat{\oh}_{Y,y}$ at a closed point $y \in Y$ of a good moduli space is simply the invariants of the induced action of $G_x$ on the miniversal deformation space $\Spf A$ of $x \in |\cX|$.

We also establish that the theorem on formal functions holds for good moduli spaces; see Theorem \ref{fns_thm}.  This provides further evidence that good moduli spaces behave very similar to proper morphisms: good moduli spaces are universally closed and finite type, preserve coherence under push forward and satisfy the formal functions theorem but are not necessarily separated.

In Section \ref{formal_git_section}, we develop the geometric invariant theory for quotients of formal affine schemes by linearly reductive group schemes. 

A sufficiently powerful structure theorem for Artin stacks giving \'etale charts by quotient stacks could imply existence of good moduli spaces Zariski-locally.  We recall the conjecture from \cite{alper_quotient}:

\begin{conj}
If $\cX$ is an Artin stack finite type over $\Spec k$ and $x \in \cX(k)$ has linearly reductive stabilizer, then there is an algebraic space $X$ over $\Spec k$ with an action of the stabilizer $G_x$, a point $\tilde x \in X$, and an \'etale morphism $[X/G_x] \arr \cX$ inducing an isomorphism $G_{\tilde x} \iso G_x$.
\end{conj}

If the conjecture is true for $x \in \cX(k)$ with the additional requirement that $X$ is affine, then there is an induced diagram
$$\xymatrix{
\cW = [X/G_x] \ar[d]^{\varphi} \ar[r]^{\qquad f}	& \cX \\
W	& ,
}$$
where $\varphi$ is a good moduli space, $f$ is an \'etale, representable morphism, and there is a point $w \in \cW(k)$ with $f(w) = x$ inducing an isomorphism $\Aut_{\cW(k)}(w) \arr \Aut_{\cX(k)} (x)$.  This is not enough to prove directly that there exists a good moduli space Zariski-locally (see Remark \ref{stab_hyp_remark}).  This leads to the natural question of what additional hypotheses need to be placed on a morphism $f: \cW \arr \cX$ where $\cW$ admits a good moduli space to imply that $\cX$ admits a good moduli space.  We prove the following theorem in Section \ref{etale_construction} (see Section \ref{notation_sec} for definitions):

\begin{thm} \label{etale_existence_thm}
Let $\cX$ be an Artin stack locally of finite type over an excellent base $S$.  Suppose there exists an \'etale, surjective, pointwise stabilizer preserving and universally weakly saturated morphism $f: \cX_1 \arr \cX$ such that there exist a good moduli space $\phi_1: \cX_1 \arr Y_1$. Then there exists a good moduli space $\phi: \cX \arr Y$ inducing a cartesian diagram
$$\xymatrix{
\cX_1 \ar[r]^f \ar[d]^{\phi_1}	& \cX \ar[d]^{\phi} \\
Y_1 \ar[r]					& Y.
}$$
\end{thm}

We offer an application of this theorem proving that the existence of a good moduli space only depends on the reduced structure (see Corollary \ref{red_cor}).  

This theorem may be of use in practice to prove existence of good moduli spaces for certain Artin stacks which can be shown to admit \'etale presentations as quotient stacks.  Conversely, if we assume that there exists a good moduli space $\cX \arr Y$, then one might hope to show the local quotient conjecture is true by showing that \'etale locally on $Y$, $\cX$ is a quotient stack by the stabilizer.

\subsection*{Acknowledgments}   I thank Max Lieblich, Martin Olsson, and Ravi Vakil for useful suggestions. 

\section{Notation} \label{notation_sec}

We will assume schemes and algebraic spaces to be quasi-separated.   We will work over a fixed base scheme $S$.  An Artin stack over $S$, in this paper, will have a quasi-compact and separated diagonal.

\subsection*{Good moduli spaces}

We recall the following two definitions and their essential properties from \cite{alper_good_arxiv}.

\begin{defn} (\cite[Definition 3.1]{alper_good_arxiv}) A morphism $f: \cX \arr \cY$ of Artin stacks is \emph{cohomologically affine} if $f$ is quasi-compact and the push-forward functor on quasi-coherent sheaves
$$f_*: \QCoh(\cX) \arr \QCoh(\cY) $$
is exact.  We say that an Artin stack $\cX$ is \emph{cohomologically affine} if the morphism $\cX \to \Spec \ZZ$ is cohomologically affine. 
\end{defn}

If $f: \cX \arr \cY$ is a representable morphism of Artin stacks where $\cY$ has quasi-affine diagonal, then $f$ is cohomologically affine if and only if $f$ is affine.  Cohomologically affine morphisms are stable under composition and base change (if the target has quasi-affine diagonal) and are local on the target under faithfully flat morphisms.  The above and further properties appear in \cite[Section 3]{alper_good_arxiv}.

\begin{defn} (\cite[Definition 4.1]{alper_good_arxiv})  \label{defn_good}
A morphism $\phi: \cX \arr Y$, with $\cX$ an Artin stack and $Y$ an algebraic space, is a \emph{good moduli space} if:
\begin{enumeratei}
\item $\phi$ is cohomologically affine.
\item The natural map $\oh_Y \iso \phi_* \oh_{\cX}$ is an isomorphism of sheaves.
\end{enumeratei}
\end{defn}

\begin{remark}  If $\cX$ is a cohomologically affine Artin stack, then the natural morphism $\cX \to \Spec \Gamma(\cX, \oh_{\cX})$ is a good moduli space.
\end{remark}

If $\phi: \cX \arr Y$ is a good moduli space, then $\phi$ is surjective, universally closed, universally submersive, and has geometrically connected fibers \cite[Theorem 4.16]{alper_good_arxiv}.  If $\cX$ is locally noetherian, then $\phi: \cX \to Y$ is universal for maps to algebraic spaces \cite[Theorem 6.6]{alper_good_arxiv}.  They are stable under arbitrary base change on $Y$ and are local in the fpqc topology on $Y$ \cite[Proposition 4.7]{alper_good_arxiv}.  Furthermore, they satisfy the strong geometric property that if $\cZ_1, \cZ_2 \subseteq \cX$ are closed substacks, then scheme-theoretically $\im \cZ_1 \cap \im \cZ_2 = \im (\cZ_1 \cap \cZ_2)$ \cite[Theorem 4.16(iii)]{alper_good_arxiv}.  This implies that for an algebraically closed $\oh_S$-field $k$, there is a bijection between isomorphism classes of objects in $\cX(k)$ up to closure equivalence and $k$-valued points of $Y$ (i.e., for points $x_1,x_2: \Spec k \arr \cX$, $\phi(x_1) = \phi(x_2)$ if and only if $\overline{ \{x_1\}} \cap \overline{ \{x_2\} } \ne \emptyset$ in $\cX \times_S k$).  Furthermore, we have the following generalization of Hilbert's 14th Problem:  if $S$ is an excellent scheme and $\cX$ is finite type over $S$, then $Y$ is finite type over $S$ \cite[Theorem 4.16(xi)]{alper_good_arxiv}.

\subsection*{Stabilizer preserving morphisms}
If $\cX$ is an Artin stack over $S$, recall that the inertia stack is defined as the fiber product
$$\xymatrix{
I_{\cX} \ar[r] \ar[d]	& \cX \ar[d]^{\Delta} \\
\cX \ar[r]^{\Delta}				& \cX \times_S \cX ,
}$$
where $\Delta : \cX \to \cX \times_S \cX$ is the diagonal.  We quickly recall the following definition introduced in \cite{alper_quotient}:

\begin{defn} Let $f: \cX \arr \cY$ be a morphism of Artin stacks.  We define:
\begin{enumeratei}
\item $f$ is \emph{stabilizer preserving} if the induced $\cX$-morphism $\psi: I_{\cX} \arr I_{\cY} \times_{\cY} \cX$ is an isomorphism.  
\item For $\xi \in |\cX|$, $f$ is \emph{stabilizer preserving at $\xi$} if for a (equivalently any) geometric point $x: \Spec k \arr \cX$ representing $\xi$, the fiber $\psi_x: \Aut_{\cX(k)}(x) \arr \Aut_{\cY(k)} (f(x))$ is an isomorphism of \emph{group schemes} over $k$.
\item $f$ is \emph{pointwise stabilizer preserving} if $f$ is stabilizer preserving at $\xi$ for all $\xi \in |\cX|$.
\end{enumeratei}
\end{defn}

\begin{remark} 
Any morphism of algebraic spaces is stabilizer preserving and any pointwise stabilizer preserving morphism is representable.  It is easy to see that both properties are stable under composition and base change.  While a stabilizer preserving morphism is clearly pointwise stabilizer preserving, the converse is not true as the following example illustrates.
\end{remark}

\begin{example} \label{example}
The following example shows that it is too much to hope for that every Artin stack Zariski-locally admits a good moduli space around a closed point with linearly reductive stabilizer.   Let $X$ be the non-separated plane attained by gluing two planes $\AA^2 = \Spec k[x,y]$ along the open set $\{x \ne 0\}$.  The action of $\ZZ_2$ on $\Spec k[x,y]_x$ given by $(x,y) \mapsto (x,-y)$ extends to an action of $\ZZ_2$ on $X$ by swapping and flipping the axis.  Then $\cX = [X / \ZZ_2]$ is a non-separated Deligne-Mumford stack.  Rydh shows in \cite[Example 7.15]{rydh_quotients} that there is no neighborhood of the origin of this stack that admits a morphism to an algebraic space which is universal for maps to schemes.  In particular, there cannot exist a neighborhood of the origin which admits a good moduli space. 
\end{example}

\subsection*{Weakly saturated morphisms}

We also recall the notion of a weakly saturated morphism which was introduced in \cite{alper_quotient}.  This notion is an essential ingredient in determining when good moduli spaces can be glued \'etale locally (see Theorem \ref{etale_existence_thm}).

\begin{defn}  
A morphism $f: \cX \arr \cY$ of Artin stacks over an algebraic space $S$ is \emph{weakly saturated} if for every geometric point $x: \Spec k \arr \cX$ with $x \in |\cX \times_S k|$ closed, the image $f_s(x) \in |\cY \times_S k|$ is closed.  A morphism $f: \cX \arr \cY$ is \emph{universally weakly saturated} if for every morphism of Artin stacks $\cY' \arr \cY$, $\cX \times_{\cY} \cY' \arr \cY'$ is weakly saturated.   
\end{defn}

\begin{remark} Although the above definition seems to depend on the base $S$, it is in fact independent:  if $S \arr S'$ is any morphism of algebraic spaces then $f$ is weakly saturated over $S$ if and only if $f$ is weakly saturated over $S'$.  Any morphism of algebraic spaces is universally weakly saturated.  If $f: \cX \arr \cY$ is a morphism of Artin stacks \emph{of finite type over $S$}, then $f$ is weakly saturated if and only if for every geometric point $s: \Spec k \arr S$, $f_s$ maps closed points to closed points.  If $f: \cX \arr \cY$ is a morphism of Artin stacks of finite type over $\Spec k$, then $f$ is weakly saturated if and only if $f$ maps closed points to closed points.
\end{remark}

\begin{remark}
The notion of weakly saturated is not stable under base change.  Consider the two different open substacks $\cU_1, \cU_2 \subseteq [\PP_1 / \GG_m]$ isomorphic to $[\AA^1 / \GG_m]$ over $\Spec k$.  Then 
$$\xymatrix{
\cU_1 \sqcup \cU_2 \sqcup \Spec k \sqcup \Spec k	\ar[r] \ar[d]		& \cU_1 \sqcup \cU_2 \ar[d] \\
\cU_1 \sqcup \cU_2 \ar[r]									& [\PP_1 / \GG_m]
}$$
is 2-cartesian and the induced morphisms $\Spec k \arr \cU_i$ are open immersions which are not weakly saturated.  This example shows that even \'etale, stabilizer preserving, surjective, weakly saturated morphisms may not be stable under base change by themselves which indicates that the \emph{universally weakly saturated} hypothesis in Theorem \ref{etale_existence_thm} is necessary.
\end{remark}

\section{Good moduli spaces for formal schemes}  \label{formal_good_sec}

In this section, we show that the theory of good moduli spaces carries over to the formal setting.  We will avoid using formal Artin stacks and make all statements and arguments using smooth, adic pre-equivalence relations.  We will also only consider the case where the good formal moduli spaces are formal schemes which suffices for our applications.  The theory of formal algebraic spaces has only been developed in the separated and locally noetherian case.  In Theorem \ref{formal_good_thm}, the noetherianness of the quotient should follow from the noetherian property of $\fU$ and the properties of good moduli spaces rather than being implicitly assumed.  Our main interest is in the case where the groupoid is induced from the inclusion of a residual gerbe of a closed point $\cG_{\xi} \hookarr \cX$ so that, in particular, the $Y_i$'s (to be defined below) are Artinian (dimension 0 noetherian schemes) and the formal good moduli space $\fY = \dlim Y_i$ is a formal affine scheme whose underlying topological space is a point.  

\subsection{Setup}  \label{subsection-setup}
 We begin by setting up the notation and making elementary remarks.

 \subsubsection{}
A \emph{smooth, adic formal $S$-groupoid} consists of source and target morphisms $\fs, \ft: \fR \rrarrows \fU$ of locally noetherian, separated formal algebraic spaces which are smooth and adic, an identity morphism $\fe: \fU \arr \fR$, an inverse $\fid: \fR \arr \fR$, and a composition $\fc: \fR \times_{\fs, \fU, \ft} \fR \arr \fR$ satisfying the usual relations.  If $\fJ$ is an ideal of definition of $\fU$, then $\fI := \fs^* \fJ$ is an ideal of definition of $\fR$ (since $\fs$ is adic), we set $U_n$ and $R_n$ to be the closed sub-algebraic spaces defined by $\fJ^{n+1}$ and $\fI^{n+1}$, respectively.  There are induced smooth $S$-groupoids  $s_n,t_n: R_n \rrarrows U_n$  with identity $e_n: U_n \arr R_n$, an inverse $i_n: R_n \arr R_n$, and a composition $c_n: R_n \times_{s_n,U_n,t_n} R_n \arr R_n$.  Set $\cX_n = [U_n / R_n]$.  Note that by \cite[Prop 3.9(iv)]{alper_good_arxiv} $\cX_n$ is cohomologically affine if and only if $\cX_0$ is.   

Let $\cX_n = [U_n / R_n]$ and suppose $\phi_n: \cX_n \arr Y_n$ is a good moduli space where $Y_n$ is a scheme for each $n$. Let $q_n: U_n \arr Y_n$ be the composite of $U_n \to \cX_n$ with $\phi_n: \cX_n \to Y_n$.   Since each $(\phi_{n})_*$ is exact, the induced map $\Gamma(Y_{n+1}, \oh_{Y_{n+1}}) \to \Gamma(Y_n, \oh_{Y_n})$ is surjective so there are closed immersions $Y_n \arr Y_{n+1}$.  The closed immersion $\cX_0 \hookarr \cX_n$ is defined by a coherent sheaf of ideals $\cI$ on $\cX_n$ such that $\cI^{n+1} = 0$.  The closed immersion $Y_0 \hookarr Y_n$ is defined by $\phi_* \cI$, which is nilpotent since $(\phi_* \cI)^{n+1} \subseteq \phi_* (\cI^{n+1}) = 0$.  It follows from \cite[I.10.6.3]{ega} that there exists a formal scheme $\fY = \dlim Y_i$ and that there is an induced morphism $\fq: \fU \arr \fY$.  We have the diagram:
\begin{equation} \label{formal_diagram}
\xymatrix{
R_0 \ar[r] \ar@<-.5ex>[d] \ar@<.5ex>[d] 		& R_1 \ar[r] \ar@<-.5ex>[d] \ar@<.5ex>[d] 	& \cdots \ar[r] 	& \fR \ar@<-.5ex>[d] \ar@<.5ex>[d] \\
U_0 \ar[r] \ar[d]		& U_1 \ar[r] \ar[d]	& \cdots \ar[r] 	& \fU   \ar[dd]^{\fq}\\
\cX_0 \ar[r] \ar[d]		& \cX_1 \ar[r] \ar[d] 		& \cdots\\
Y_0 \ar[r]				& Y_1 \ar[r]			& \cdots \ar[r]	& \fY,
}
\end{equation}
where all appropriate squares are 2-commutative and the appropriate squares in the top and middle rows are 2-cartesian.  Note that the squares in the bottom row are not necessary cartesian.  There should exist a geometric object $\h{\cX}$ (i.e., a formal Artin stack) filling in the above diagram for which $\fq$ factors through.

We note that the formal scheme $\fY$ and the morphism $\fq: \fU \arr \fY$ do not depend on the choice of the ideal of definition.

We do not know a priori that $\fY$ is locally noetherian.  In particular, if each $Y_i = \Spec A_i$ is an affine scheme, it is not immediate that the topological ring $\ilim A_i$ is either adic or noetherian.

\subsubsection{} \label{groupoid_start}
Recall that $\fq$ denotes the morphism $\fq: \fU \to \fY$.  There is a natural map $\oh_{\fY} \arr (\fq_* \oh_{\fU})^{\fR}$, where $(\fq_* \oh_{\fU})^{\fR}$ denotes the sheaf of topological rings on $\fY$ which assigns to an open $V \subseteq \fY$,  the equalizer
$$\oh_{\fU} (\fq^{-1}(V)) \rrarrows \oh_{\fR} ( (\fq \circ \ft)^{-1} (V));$$
clearly $\oh_{\fY}(V) \arr \oh_{\fU}(\fq^{-1}(V))$ factors through this equalizer as $\fq \circ \fs = \fq \circ \ft$.  

\subsubsection{} More generally, if $\fF$ is a coherent $\oh_{\fU}$-module, an \emph{$\fR$-action} on $\fF$ is an isomorphism $\alpha: \fs^* \fF \arr \ft^* \fF$ satisfying the usual cocycle condition on $\fR \times_{t,\fU, s} \fR$.  If $F_n$ denotes the pullback of $\fF$ to $U_n$, then $F_n$ inherits a $R_n$-action and therefore descends to a coherent sheaf $\cF_n$ of $\oh_{\cX_n}$-modules.  We will denote by $(\fq_* \fF)^{\fR}$ the sheaf of $\oh_{\fY}$-modules defined by the equalizer
$$\xymatrix{
\fq_* \fF \ar@<.5ex>[r]^{\ft^*} \ar@<-.5ex>[r]_{\alpha \circ \fs^*}		& (\fq \circ \ft)_* \ft^* \fF.
}$$
If there were a formal stack $\h \cX$, then $(\fq_*\fF)^{\fR}$ should simply be the push forward under $\h \cX \arr \fY$ of the descended sheaf of $\oh_{\h \cX}$-modules $\h \fF$.  We also write $\Gamma(\fU, \fF)^{\fR} = \Gamma(\fU, (\fq_* \fF)^{\fR})$.
 
It is not obvious that $(\fq_* \fF)^{\fR}$ is coherent but we will show in Theorem \ref{formal_good_thm} that this is true if $Y_0$ is Artinian.   The morphisms $(\fq_* \fF)^{\fF} \arr ((q_i)_* F_i)^{R_i} = (\phi_i)_* \cF_i$ induces a morphism of $\oh_{\fY}$-modules
\begin{equation} \label{sheaf_mor}
(\fq_*\fF)^{\fR} \arr \ilim (\phi_i)_* \cF_i. 
\end{equation}
\subsubsection{}  If $\fI$ is a coherent sheaf of ideals in $\oh_{\fU}$, we say that $\fI$ is \emph{$\fR$-invariant} if $\fs^*\fJ = \ft^* \fJ $.  The sheaf $\fI$ therefore inherits an $\fR$-action.  We say that a closed sub-algebraic space $\fZ \subseteq \fU$ is \emph{$\fR$-invariant} if it is defined by an invariant sheaf of ideals.

\subsubsection{} \label{groupoid_end}
For any adic morphism of formal schemes $\fY' \arr \fY$, by taking fiber products, there is an induced diagram as in diagram (\ref{formal_diagram}).  There are source and target morphisms $\fs', \ft': \fR' \rrarrows \fU'$, an identity morphism $\fe': \fU' \arr \fR'$, an inverse $\fid': \fR' \arr \fR'$ and a composition $\fc': \fR' \times_{\fs', \fU', \ft'} \fR' \arr \fR'$ satisfying the usual relations.  Suppose further that $\fY'$, $\fY$, and $\fU' = \fY' \times_{\fY} \fU$ are locally noetherian.  
Then $(\fs', \ft': \fR' \rrarrows \fU', \fe', \fid')$ indeed defines a smooth, adic formal $S$-groupoid.  Because good moduli spaces are stable under arbitrary base change, there are good moduli spaces $\phi_i': \cX_i' \arr Y_i'$.  Furthermore, the induced morphisms $\dlim U'_i \arr \fU'$, $\dlim R'_i \arr \fR'$, and $ \dlim Y'_i \arr \fY'$ are isomorphisms. 

\subsection*{Formal good moduli spaces}

\begin{thm} \label{formal_good_thm}
Assume the notation above. 
\begin{enumeratei}
\item \label{formal_functions}
	The natural map $\oh_{\fY} \arr (\fq_* \oh_{\fU})^{\fR}$ is an isomorphism of sheaves of topological rings.
\item \label{formal_exact}
	The functor from coherent sheaves on $\fU$ with $\fR$-actions to sheaves on $\fY$ given by $\fF \mapsto (\fq_* \fF)^{\fR}$ is exact.  Furthermore, the morphism $(\fq_*\fF)^{\fR} \arr \ilim (\phi_i)_* \cF_i$ is an isomorphism of topological $\oh_{\fY}$-modules.

\item \label {formal_sur}
	$\fq$ is surjective.

\item \label{formal_closed}
	If $\fZ \subseteq \fU$ is a closed, $\fR$-invariant formal subscheme, then $\fq(\fZ)$ is closed, 

\item \label{formal_set_sep}
If $\fZ_1, \fZ_2 \subseteq \fU$ are closed, $\fR$-invariant formal subschemes, then set-theoretically 
	$$\fq(\fZ_1) \cap \fq(\fZ_2) = \fq(\fZ_1 \cap \fZ_2).$$

\item \label{formal_univ}	
	$\fq$ is universal for $\fR$-invariant maps to formal schemes.  That is, given a morphism $\psi: \fU \arr \fW$ where $\fW$ is a formal scheme such that $\fs \circ \psi = \ft \circ \psi$, then there exists a unique morphism $\chi: \fY \arr \fW$ such that $\chi \circ \fq = \psi$.

\item \label{formal_noeth}
If $\fY= \Spf A$ is an affine formal scheme, then $A$ is noetherian.

\end{enumeratei}
\smallskip
Suppose furthermore that $\dim Y_0 = 0$ (i.e., $Y_0$ is an Artinian scheme).
\smallskip
\begin{enumeratei} \setcounter{enumi}{7}

\item \label{formal_adic}
$\fY$ is a locally noetherian formal scheme.  In particular, if $\fY = \Spf A$ and $m = \ker(A \arr A_0)$, then $A$ is an $m$-adic noetherian ring.

\item \label{formal_coherence}
If $\fF$ is a coherent sheaf of $\fU$ with $\fR$-action, then $(\fq_* \fF)^{\fR}$ is a coherent $\fY$-module.

\item \label{formal_sep}
	If $\fI$ and $\fJ$ are two $\fR$-invariant coherent ideals in $\oh_{\fU}$, then the natural map
	$$ (\fq_* \fI)^{\fR} + (\fq_* \fJ)^{\fR} \arr (\fq_* (\fI + \fJ))^{\fR}$$
is an isomorphism.  If $\fZ_1$ and $\fZ_2$ are $\fR$-invariant formal closed subschemes, then scheme-theoretically
$$ \im \fZ_1 \cap \im \fZ_2 = \im (\fZ_1 \cap \fZ_2),$$
where $\im \fZ$ denotes the scheme-theoretic image of $\fZ$ under $\fq: \fU \arr \fY$ and is defined by the coherent sheaf of ideals $\ker(\oh_{\fY} \arr \fq_* \oh_{\fZ})$.  
\end{enumeratei}
\end{thm}

\begin{proof}  For (\ref{formal_functions}), for each $n$ we have an exact sequence
$$\oh_{Y_n} \arr (q_n)_* \oh_{U_n} \rrarrows (q_n \circ t_n)_* \oh_{R_n}.$$
By taking inverse limits, we get that $\oh_{\fY} = \ilim \oh_{Y_n}$ is naturally identified with the equalizer of $\fq_* \oh_{\fU} \rrarrows (\fq \circ \ft)_* \oh_{\fR}$, which is the definition of $(\fq_* \oh_{\fU})^{\fR}$.

For (\ref{formal_exact}), we first note that the above argument generalizes to show that the morphism (\ref{sheaf_mor}) is an isomorphism of topological $\oh_{\fY}$-modules.  Indeed, for each $n$ we have an exact sequence
$$(\phi_n)_* \cF_n \arr (q_n)_* F_n \rrarrows (q_n \circ t_n)_* t_n^* F_n,$$
and by taking inverse limits, we get that $\ilim (\phi_n)_* \cF_n$ is identified with the equalizer $\fq_* \fF \rrarrows (\fq \circ \ft)_* \ft^* \fF$.  The functor $\fF \mapsto (\fq_* \fF)^{\fR}$ is clearly left exact.  Consider a surjection $\fF \mapsonto \fG$ of coherent $\oh_{\fU}$-modules with $\fR$-action, which induces surjections $F_n \mapsonto G_n$ of coherent $\oh_{U_n}$-modules with $R_n$-action and $\cF_n \mapsonto \cG_n$ of coherent $\oh_{\cX_n}$-modules.  Since $(\phi_n)_*$ is exact, we have that $(\phi_n)_* \cF_n \mapsonto (\phi_n)_* \cG_n$ is surjective.  Furthermore, the inverse system $((\phi_n)_* \cG_n)$ is Mittag-Leffler (i.e., $(\phi_{n+1})_* \cG_{n+1} \mapsonto (\phi_n)_* \cG_n$ is surjective) since $\phi_{n+1}$ is exact.  Therefore, 
$$\ilim (\phi_n)_* \cF_n \mapsonto \ilim (\phi_n)_* \cG_n$$  
is surjective and is identified with $(\fq_* \fF)^{\fR} \mapsonto (\fq_* \fG)^{\fR}$.

Since properties (\ref{formal_sur}), (\ref{formal_closed}), and (\ref{formal_set_sep}) are topological, they follow directly from the corresponding property for good moduli spaces (\cite[Theorem 4.16(i),(ii) and (iii)]{alper_good_arxiv}).

For (\ref{formal_univ}), the argument of \cite[Proposition 0.1 and Remark (5) on p.\ 8]{git} adapts to this setting as in \cite[Theorem 4.15(vi)]{alper_good_arxiv}.

For (\ref{formal_noeth}), let $I \subseteq A$ be an ideal.  Let $I_n = \pi_n(I) \subseteq A_n$ where $Y_n = \Spec A_n$ and $\pi_n: A \mapsonto A_n$.  The closed sub-algebraic space 
$$U'_n = U_n \times_{\Spec A_n} \Spec A_n / I_n \hookarr U_n$$
 is defined by the sheaf of ideals $J_n = q_n^* {\tilde I_n} \cdot \oh_{U_n}$.  Then $\fU' = \dlim U'_n$ is the closed formal sub-algebraic space of $\fU$ defined by the coherent sheaf of ideals $\fJ = \ilim J_n$.  The sheaf $J_n$ is $R_n$-invariant descending to a coherent sheaf of ideals $\cJ_n$ in $\oh_{\cX_n}$.  By \cite[Lemma 4.12]{alper_good_arxiv}, $I_n \arr \Gamma(\cX_n, \cJ_n)$ is an isomorphism and therefore by part (\ref{formal_exact}), in the diagram
$$\xymatrix{
I \ar[d] \ar[r]					& \ilim I_n \ar[d] \\
\Gamma(\fU, \fJ)^{\fR} \ar[r] 		& \ilim \Gamma(\cX_n, \cJ_n),
}$$
the bottom row is an isomorphism.  It follows that the left vertical arrow is an isomorphism.  Since $\fU$ is noetherian, it follows that any ascending chain $I^{(1)} \subseteq I^{(2)} \subseteq \cdots$ of ideals in $A$ terminates.

For (\ref{formal_adic}) and (\ref{formal_coherence}), we may assume $\fY = \Spf A$ where $A$ is a noetherian ring by (\ref{formal_noeth}).  We must show that $A$ is an adic ring.  Let $I_n = \ker (A \arr A_n)$.  Clearly, $I_n \supseteq I_0^n$.  Since $A/I_0^n$ is Artinian, the descending chain $I_0 \supseteq I_1 \supseteq \cdots$ of ideals in $A/I_0^n$  terminates so that there exists $k$ such that $I_0^n \supseteq I_k$.  This implies that $I_0^n$ is open so that $A$ is $I_0$-adic.  Similarly, $M = \Gamma(\fU, \fF)^{\fR} = \ilim \Gamma(\cX_i, \cF_i)$ is Hausdorff and complete with respect to the $I_0$-adic topology.  It follows from \cite[0.7.2.9]{ega} that $M$ is a finitely generated $A$-module.

For (\ref{formal_sep}), we have the identifications $(\fq_* \fI)^{\fR} = \ilim (\phi_n)_* \cI_n$, $(\fq_* \fJ)^{\fR} = \ilim (\phi_n)_* \cJ_n$ and $\fq_* (\fI+ \fJ)^{\fR} = \ilim (\phi_n)_* (\cI_n + \cJ_n)$ where $\cI_n$ and $\cJ_n$ are the corresponding sheaf of ideals on $\cX_n$.  For each $n$, by \cite[Lemma 4.9]{alper_good_arxiv}, the inclusion $(\phi_n)_* \cI_n + (\phi_n)_* \cJ_n \arr (\phi_n)_* (\cI_n + \cJ_n)$ is an isomorphism.  By taking inverse limits,
$$\ilim ((\phi_n)_* \cI_n + (\phi_n)_* \cJ_n) \arr \ilim  (\phi_n)_* (\cI_n + \cJ_n)$$
is an isomorphism.  Since
$$ \ilim (\phi_n)_* \cI_n + \ilim (\phi_n)_* \cJ_n \arr \ilim ((\phi_n)_* \cI_n + (\phi_n)_* \cJ_n)$$
is also an isomorphism, we have that  $(\fq_* \fI)^{\fR} + (\fq_* \fJ)^{\fR} \arr (\fq_* (\fI + \fJ))^{\fR}$ is an isomorphism.  The final statement follows from the identification of the coherent sheaf of ideals $(\fq_* \fI)^{\fR}$ with $\ker(\oh_{\fY} \arr \fq_* \oh_{\fZ})$. \end{proof}

\begin{remark}  As in \cite{alper_good_arxiv}, we contend that properties (i) and (ii) should in fact define the notion of a \emph{formal good moduli space} and these two properties alone should imply the others.  However, this theory would best be developed in the language of formal stacks which we are avoiding in this paper.
\end{remark}

\subsection*{Groupoids induced from closed substacks}
Let $\cX$ be a noetherian Artin stack and $\cZ$ be a closed substack which is cohomologically affine (i.e., that this means that $\cZ \to \Spec \ZZ$ is cohomologically affine).  Then $\cZ$ together with a presentation $U \arr \cX$ induces a smooth, adic formal $S$-groupoid and a diagram as in (\ref{formal_diagram}).  Let $\cX_0 = \cZ$ and $\cX_n$ is the closed substack corresponding to the $n$-th nilpotent thickening.  Set $U_i = U \times_{\cX} \cX_i$ and $R_i = R \times_{\cX} \cX_i$.  Then the smooth $S$-groupoids $R_i \rrarrows U_i$ induces the smooth, adic formal $S$-groupoid $\fR \rrarrows \fU$ where $\fU = \dlim U_i$ and $\fR = \dlim R_i$ (with the source, target, identity, inverse and composition morphisms defined in the obvious way).  
 
 Since $\cX_0$ is cohomologically affine, its nilpotent thickenings $\cX_n$ are also cohomologically affine.  Therefore, there are good moduli spaces $\phi_n: \cX_n \arr Y_n$.  If $\fY = \dlim Y_i = \Spec \ilim \Gamma(\cX_n, \oh_{\cX_n}) $, there is an induced $\fR$-invariant morphism $\fq: \fU \arr \fY$ and we can apply the above theorem to conclude the following:

\begin{cor} \label{local_good_cor}
Suppose $\cZ $ is a closed, cohomologically affine substack of a noetherian Artin stack $\cX$ such that $\Gamma(\cZ, \oh_{\cZ})$ is Artinian.  Then with the notation above, there is an induced morphism $\fq: \fU \arr \fY$ satisfying the properties {\rm (i)} through {\rm (x)} in Theorem $\ref{formal_good_thm}$. 
\epf
\end{cor}

The corollary above implies that there is an isomorphism of topological rings 
$$\ilim \Gamma(\cX_n, \oh_{\cX_n}) \arr (\ilim \Gamma(U_n, \oh_{U_n}))^{\fR}.$$
   If there exists a good moduli space $ \cX \arr Y$, it is natural to compare these topological rings with the complete local ring induced by the image of $\cZ$.

\begin{prop}  \label{proposition-formal}
Suppose $\cX$ is a locally noetherian Artin stack admitting a good moduli space $\phi: \cX \arr Y$ and $\cZ \subseteq \cX$ is a closed substack defined by a sheaf of ideals $\cI$.  Let $\cX_n$ be the nilpotent thickenings of $\cZ$ defined by $\cI^{n+1}$.  If $\cZ \subseteq \cX$ is cohomologically affine and $\Gamma(\cZ, \oh_{\cX})$ is Artinian, then the image $y \in |Y|$ of $\cZ$ is a closed point and the induced morphism
$$ \h{\oh}_{Y,y} \arr \ilim \Gamma(\cX_n, \oh_{\cX_n})$$
is an isomorphism, where $\h{\oh}_{Y,y} = \ilim \Gamma(Y, \oh_Y / \cJ^n)$ and $\cJ$ defines the closed immersion $\Spec k(y) \hookarr Y$.  
\end{prop}

\begin{proof}  We have that $\phi_* \cI \subseteq \cJ$ and $\ilim \Gamma(Y, \oh_Y/ (\phi_* \cI)^n) \arr \h{\oh}_{Y,y}$ is an isomorphism.  We also have the identification $\ilim \Gamma(\cX_n, \oh_{\cX_n}) = \ilim \Gamma(Y, \phi_* \cI^n)$.  There is an inclusion $(\phi_*\cI)^n \subseteq \phi_* (\cI^n)$.  Since $Y_n$ is Artinian, the descending chain of sheaves of ideals $\phi_*(\cI^n) \supseteq \phi_*(\cI^{n+1}) \supseteq \cdots$ in $Y_n$ terminates so that for all $n$, there exists a $N$ such that $\phi_*(\cI^N) \subseteq (\phi_* \cI)^n$.
\end{proof}

\subsection*{Local structure around closed points with linearly reductive stabilizer} 
We apply the results above to the case in which we are most interested in:   $\cX$ is a noetherian Artin stack and $\xi \in |\cX|$ is a closed point with linearly reductive stabilizer.  Let $\cG_{\xi}$ be the residual gerbe of $\xi$ (see \cite[Section 11]{lmb}).
There is a closed immersion $\cG_{\xi} \hookarr \cX$ which, as in (\ref{subsection-setup}), induces a smooth, adic formal $S$-groupoid $\fR \rrarrows \fU$.

Since $\xi \in |\cX|$ has linearly reductive stabilizer (see \cite[Definition 12.12]{alper_good_arxiv}), $\cG_{\xi}$ is cohomologically affine and $\phi_0: \cG_{\xi} \arr \Spec k(\xi)$ is a good moduli space.  The nilpotent thickenings also admit good moduli spaces $\phi_n: \cX_n \arr Y_n$ and there is an induced morphism $\fq: \fU \arr \fY$.

\begin{cor} \label{local_point_cor} Suppose $\xi \in |\cX|$ is a closed point with linearly reductive stabilizer.  Then with the above notation, there is an induced morphism $\fq: \fU \arr \fY$ satisfying the properties {\rm(i)} through {\rm (x)}
in Theorem $\ref{formal_good_thm}$. \epf
\end{cor}

In particular, Corollary \ref{local_point_cor} implies that there is an isomorphism of topological rings $\ilim \Gamma(\cX_n, \oh_{\cX_n}) \arr (\ilim \Gamma(U_n, \oh_{U_n}))^{\fR}.$  There may not exist a good moduli space for $\cX$ but Theorem \ref{theorem-formal-local} establishes that we do in fact know the local structure of the good moduli space if it exists.  

\medskip \noindent
{\it Proof of Theorem $\ref{theorem-formal-local}$. }  Note that the stabilizer $G_x$ is linearly reductive since $x \in |\cX|$ is a closed point. The existence of quotient stack presentations $\cX_i \cong [\Spec A_i / G_x]$ follows from \cite[Theorem 1]{alper_quotient}.  The theorem then follows from Proposition \ref{proposition-formal} and Corollary \ref{local_point_cor}. \epf

\begin{remark}  With the notation of Theorem \ref{theorem-formal-local}, if $x \in \cX(k)$ is not a closed point, then not much can be said about the local structure of $Y$ around $\phi(x)$; even the dimensions of the good moduli spaces may vary as one varies open substacks containing $x$.  For instance, consider $\GG_m \times \GG_m$ acting on $\AA^4$ via $(t,s) \cdot (w,x,y,z) = (tw, tx, sy, sz)$.  Let $\cX =  [\AA^4 / \GG_m \times \GG_m]$ and  $x = (1,1,1,1) \in \cX$.  Let $\cU$ be the open locus where $(w,x) \ne (0,0)$ and $\cV \subseteq \cU$ be the sub-locus where $(y,z) \ne (0,0)$.  Then we have a commutative diagram of good moduli spaces of open substacks containing $x$
$$\xymatrix{
\cV \ar[d] \ar[r]				& \cU \ar[d] \ar[r]	& \cX \ar[d] \\
\PP^1 \times \PP^1 \ar[r]	 	& \PP^1 \ar[r] 	 	& \Spec k.
}$$
\end{remark}

\subsection*{The theorem on formal functions}
Let $\cX$ be a noetherian Artin stack and let $\phi: \cX \arr Y$ be a good moduli space.  Let $Y_0 \subseteq Y$ be a closed subscheme defined by a sheaf of ideals $\cJ$ and set $\cX_0:=\phi^{-1}(Y') \subseteq |\cX|$ which is defined by $\cI: = \phi^* \cJ \cdot \oh_{\cX}$.  Let $Y_k$ be the $k$-th nilpotent thickening of $Y_0$ defined by $\cJ^{k+1}$ and $\cX_k = \cX \times_Y Y_k$ the $k$-th nilpotent thickening of $\cX_0$.  

If $\cF$ is a coherent sheaf of $\oh_{\cX}$-modules, set $\cF_k = \cF / \cI^{k+1} \cF$.  For a coherent sheaf $\cG$ of $\oh_{Y}$-modules, let $\hat{\cG} := \ilim \cG / \cJ^{k+1} \cG$

\begin{remark} \label{exactness_remark}
 If we let $U \arr \cX$ be a presentation, then as above there is an induced smooth, adic formal $S$-groupoid $\fR \rrarrows \fU$.  Let $\fY = \dlim Y_n$ and let $\fq: \fU \arr \fY$ be the induced morphism.  A coherent $\oh_{\cX}$-module $\cF$ induces a coherent $\oh_{\fU}$-module $\fF$ with an $\fR$-action.  As we saw in Theorem \ref{formal_good_thm}, there is an isomorphism of topological $\oh_{\fY}$-modules
$$(\fq_*\fF)^{\fR} \arr \ilim \phi_* \cF_k.$$
Since the functor $\cF \mapsto \fF$ is exact and by Theorem \ref{formal_good_thm} the functor $\fF \mapsto (\fq_* \fF)^{\fR}$ is exact, it follows that the functor $\cF \mapsto \ilim \phi_* \cF_k$ is exact.
\end{remark}

\begin{thm} \label{fns_thm} Let $\cX$ be a noetherian Artin stack, $\phi: \cX \arr Y$ a good moduli space and $Y_0 \subseteq Y$ a closed sub-algebraic space.  If $\cF$ is a coherent $\oh_{\cX}$-module, for each $n \ge 0$, the natural morphism
$$\hat{R^n \phi_* (\cF)} \arr  \ilim R^n \phi_* (\cF_k)$$
is an isomorphism.
\end{thm}

\begin{proof}  We may assume $Y$ is a scheme.  Because $\phi_*$ is exact, the case of positive $n$ is obvious and we must only show that 
$$\hat{\phi_* \cF} \arr \ilim \phi_* \cF_k$$
is an isomorphism.  Define $\cK$ and $\cL$ by the exact sequence
$$0 \arr \cK \arr \phi^* \phi_* \cF \arr \cF \arr \cL \arr 0.$$
Since $\phi_*$ and completion are exact functors and, by the above remark, $\cF \mapsto \ilim \phi_* \cF_k$ is exact, we have a commutative diagram
$$\xymatrix{
0 \ar[r]	& \hat{\phi_* \cK} \ar[r] \ar[d]	& \hat{\phi_* \phi^* \phi_* \cF} \ar[r] \ar[d]	& \hat{\phi_* \cF} \ar[r] \ar[d] & \hat{\phi_* \cL} \ar[r] \ar[d]		& 0\\
0 \ar[r]	& \ilim \phi_* \cK_k \ar[r] & \ilim \phi_* (\phi^* \phi_* \cF)_k \ar[r] &\ilim \phi_* \cF_k \ar[r]&\ilim \phi_* \cL_k \ar[r]& 0&
}$$
with both rows exact.  We note that $\phi_* \cK = \phi_* \cL = 0$ and since $\phi_* \cK \mapsonto \phi_* \cK_k$ and $\phi_* \cL \mapsonto \phi_* \cL_k$ are surjective, it follows that $\ilim \phi_* \cK_k = \ilim \phi_* \cL_k = 0$.  Therefore, it suffices to prove the theorem in the case that $\cF=\phi^* \cG$ is the pullback of a coherent sheaf $\cG$ on $Y$.  In this case, $\hat{ \phi_* \cF} = \hat{\cG}$ and $\phi_* \cF_k = \cG / \cI^{k+1} \cG$, and the statement is clear.
\end{proof}

By applying the above theorem when $Y_0$ is a point and $Y$ is affine, we obtain the following corollary.
\begin{cor}  
 Let $\cX$ be a noetherian Artin stack, $\phi: \cX \arr Y$ a good moduli space with $Y$ affine and $y \in Y$ a point.  If $\cF$ is a coherent $\oh_{\cX}$-module, for each $n \ge 0$, the natural morphism
$$\hat{H^n(\cX, \cF)} \arr  \ilim H^n(\cX_k, \cF_k)$$
is an isomorphism. \epf
\end{cor}

\section{Geometric Invariant Theory for Formal Schemes} \label{formal_git_section}

In this section, we show that the constructions of geometric invariant theory carry over for actions of linearly reductive group schemes on formal affine schemes.

\subsection*{Setup} 
Let $G$ be a linear reductive affine group scheme over a locally noetherian scheme $S$.  Recall from \cite[Section 12]{alper_good_arxiv} that this means that $G \arr S$ is flat, finite type, and affine and the morphism $BG \arr S$ is cohomologically affine.  If $\fX$ is a locally noetherian formal scheme over $S$, an action of $G$ on $\fX$ consists of a morphism $\sigma: G \times_S \fX \arr \fX$ such that the usual diagrams commute.  Let $\fI$ be the largest ideal of definition (see \cite[0.7.1.6]{ega}).  Note that both the projection and multiplication $p_2, \sigma: G \times_S \fX \arr \fX$ are adic morphisms, and that $\fI$ is $G$-invariant.

If we denote $X_n =  (\fX, \oh_{\fX}/\fI^{n+1})$ as the closed subscheme defined by $\fI^{n+1}$, there are induced compatible actions of $G$ on $X_n$.  Conversely, given compatible actions of $G$ on the $X_n$, there is a unique action of $G$ on $\fX$ restricting to the actions on $X_n$.

Suppose further that $\fX = \Spf B$, $S = \Spec C$ with $B$ is an $I$-adic $C$-algebra and $G$ is an affine fppf linearly reductive group scheme over $S$.  The action of $G$ on $\fX$ translates into a dual action $\sigma^{\#}: B \arr \Gamma(G) \htensor_C B$ with $\sigma^{\#}(I) \subseteq \Gamma(G) \htensor I$.    The action corresponds to a compatible family of dual actions $\sigma_n^{\#}: B/I^n \arr \Gamma(G) \tensor_C B/I^n$.  Define
$$\xymatrix{
B^G =\Eq( B   \ar@<.5ex>[r]^{\sigma^{\#}} \ar@<-.5ex>[r]_{p_2^{\#}}  & \Gamma(G) \htensor_C B).}$$

Then $\sigma, p_2: G \times_S \fX \rrarrows \fX$ is a smooth, adic formal $S$-groupoid where the identity, inverse and composition morphisms and the commutativity of the appropriate diagrams are induced formally from the group action.  

The quotient stacks $\cX_n = [X_n /G]$ are cohomologically affine and therefore admit good moduli spaces $\phi_n: \cX_n \arr Y_n$ where $Y_n = \Spec (B/I^n)^G$.  Let $\fY = \dlim Y_i$ and $\fq: \fX \arr \fY$ be the induced morphism.  The observations in \ref{groupoid_start} through \ref{groupoid_end} have obvious analogues to the case of group actions.

Theorems \ref{formal_good_thm} translates into the following theorem.

\begin{thm} \label{fgit_thm}
Assume the above notation.
\begin{enumeratei}
\item \label{fgit_functions}
	The natural map $\oh_{\fY} \arr (\fq_* \oh_{\fX})^{G}$ is an isomorphism of sheaves of topological rings.
\item \label{fgit_exact}
	The functor from coherent sheaves on $\fX$ with $G$-actions to sheaves on $\fY$ given by $\fF \mapsto (\fq_* \fF)^G$ is exact.  Furthermore, the morphism $(\fq_*\fF)^G \arr \ilim (\phi_i)_* \cF_i$ is an isomorphism of topological $\oh_{\fY}$-modules.

\item \label {fgit_sur}
	$\fq$ is surjective.

\item \label{fgit_closed}
	If $\fZ \subseteq \fX$ is a closed, $G$-invariant formal subscheme, then $\fq(\fZ)$ is closed.  

\item \label{fgit_set_sep}
If $\fZ_1, \fZ_2 \subseteq \fX$ are closed, $G$-invariant formal subschemes, then set-theoretically 
	$$\fq(\fZ_1) \cap \fq(\fZ_2) = \fq(\fZ_1 \cap \fZ_2).$$

\item \label{fgit_univ}	
	$\fq$ is universal for $G$-invariant maps to formal schemes.

\item \label{fgit_noeth}
If $\fY= \Spf A$ is an affine formal scheme, then $A$ is noetherian.

\end{enumeratei}
\bigskip
Suppose furthermore that $\dim Y_0 = 0$ (i.e., $Y_0$ is an Artinian scheme).
\bigskip
\begin{enumeratei} \setcounter{enumi}{7}

\item \label{fgit_adic}
$\fY$ is a locally noetherian formal scheme.  In particular, if $\fY = \Spf A$ and $m = \ker(A \arr A_0)$, then $A$ is an $m$-adic noetherian ring.

\item \label{fgit_coherence}
If $\fF$ is a coherent sheaf of $\fX$ with $\fR$-action, then $(\fq_* \fF)^G$ is a coherent $\fY$-module.

\item \label{fgit_sep}
	If $\fI$ and $\fJ$ are two $G$-invariant coherent ideals in $\oh_{\fX}$, then the natural map
	$$ (\fq_* \fI)^{G} + (\fq_* \fJ)^{G} \arr (\fq_* (\fI + \fJ))^{G}$$
is an isomorphism.  
\end{enumeratei}
\epf
\end{thm}

\begin{remark}  The formal analogue of Nagata's fundamental lemma for linear reductive group actions (\cite{nagata_invariants-affine}) hold:  if $G$ is a linearly reductive group acting a noetherian affine formal scheme $\Spf A$, then
\begin{enumeratei}
\item for an invariant ideal $J \subseteq A$,
$$A^G / (J \cap A^G) \iso (A/J)^G,$$
\item for invariant ideal $J_1, J_2 \subseteq A$,
$$J_1 \cap A^G + J_2 \cap A^G \iso (J_1+J_2) \cap A^G.$$
\end{enumeratei}
\end{remark}

\section{\'Etale local construction of good moduli spaces} \label{etale_construction}

\subsection*{Recalling properties of good moduli spaces}
We recall the necessary results from \cite{alper_quotient} which generalize analogous results from \cite{alper_good_arxiv}.

\begin{prop} \label{etale_preserving_cor}  (\cite[Corollary 6.6]{alper_quotient}) 
Consider a commutative diagram $$\xymatrix{ 
\cX \ar[r]^f \ar[d]^{\phi}		& \cX' \ar[d]^{\phi'} \\
Y \ar[r]^g					& Y'
}$$
with $\cX, \cX'$ locally noetherian Artin stacks of finite type over $S$, $g$ locally of finite type, and $\phi, \phi'$ good moduli spaces.  If $f$ is \'etale, pointwise stabilizer preserving and weakly saturated, then $g$ is \'etale.
\end{prop}

\begin{prop}   \label{finite_prop} (\cite[Proposition 6.7]{alper_quotient})
Suppose $\cX, \cX'$ are locally noetherian Artin stacks and
$$\xymatrix{
\cX \ar[r]^f \ar[d]^{\phi}		& \cX' \ar[d]^{\phi'} \\
Y \ar[r]^{g}					& Y'
}$$
is commutative with $\phi,\phi'$ good moduli spaces.  Suppose
\begin{enumeratea}
\item $f$ is representable, quasi-finite and separated,
\item $g$ is finite,
\item $f$ is weakly saturated.
\end{enumeratea}
Then $f$ is finite. 
\end{prop}

\begin{prop} \label{square_prop} (\cite[Proposition 6.8]{alper_quotient})
Suppose $\cX, \cX'$ are locally noetherian Artin stacks and 
$$\xymatrix{
\cX \ar[r]^f \ar[d]^{\phi}		& \cX' \ar[d]^{\phi'} \\
Y \ar[r]^{g}					& Y'
}$$
is a commutative diagram with $\phi,\phi'$ good moduli spaces.  If $f$ is representable, separated, \'etale, stabilizer preserving and weakly saturated, then $g$ is \'etale and the diagram is cartesian.
\end{prop}

We prove a simple proposition which concludes that good moduli spaces exist locally near a preimage of a closed point after a quasi-finite, separated base change.

\begin{prop}  Suppose there is a diagram
$$\xymatrix{
\cX \ar[r]^f 	& \cX' \ar[d]^{\phi'} \\
			& Y'
}$$
with $f$ a representable, quasi-finite, separated morphism of locally noetherian Artin stacks and $\phi'$ a good moduli space.  Suppose $\xi \in |\cX|$ has closed image $\xi' \in |\cX'|$.  Then there exists an open substack $\cU \subseteq \cX$ containing $\xi$ and a commutative diagram
$$\xymatrix{
\cU \ar[r]^{f|_{\cU}} \ar[d]^{\phi}		& \cX' \ar[d]^{\phi'} \\
 Y \ar[r]^g							& Y'
 }$$
 with $\phi$ a good moduli space.
 \end{prop}
 
 \begin{proof}
 By applying Zariski's Main Theorem \cite[Theorem 16.5]{lmb}, there is a factorization $f: \cX \stackrel{i}{\arr} \tilde \cX \stackrel{\tilde f}{\arr} \cX'$ with $i$ an open immersion and $\tilde f$ finite.  Therefore, there is a commutative diagram
 $$\xymatrix{
 \cX \ar@{^(->}[r]^i	& \tilde \cX\ar[d]^{\tilde \phi} \ar[r]^{\tilde f}	& \cX' \ar[d]^{\phi'} \\
 			& \tilde Y \ar[r]^{\tilde g}					& Y'
 }$$
 with $\tilde \phi: \tilde \cX \arr \tilde Y := \sSpec \phi'_* \tilde f_* \oh_{\tilde \cX}$ and $\tilde g$ is finite.  Since $\tilde f$ is finite, $\xi \in \tilde \cX$ is closed.  Therefore, $\{\xi\}$ and $\cZ:=\tilde \cX \setminus \cX$ are disjoint, closed substacks so $\tilde \phi(\xi)$ and $\tilde \phi(\cZ)$ are closed and disjoint.  If $Y = \tilde Y \setminus \tilde \phi(\cZ)$, then $\cU = \tilde \phi^{-1}(Y)$ is an open substack containing $\xi$ and contained in $\cX$ admitting a good moduli space $\cU \arr Y$. 
 \end{proof}

 We can also prove that good moduli spaces satisfy effective descent for separated, \'etale, pointwise stabilizer preserving, and weakly saturated morphisms.  A version of the following proposition allows one to conclude that good moduli spaces for locally noetherian Artin stacks are universal for maps to algebraic spaces (see \cite[Theorem 6.6]{alper_good_arxiv}).
 
 \begin{prop} \label{eff_descent_prop}
Suppose $\phi': \cX' \arr Y'$ is a good moduli space and $f: \cX \arr \cX'$ is a surjective, separated, \'etale, pointwise stabilizer preserving, and weakly saturated morphism of locally noetherian Artin stacks.  Then there exists a good moduli space $\phi: \cX \arr Y$ inducing $g: Y \arr Y'$ such that the diagram
$$\xymatrix{
\cX \ar[r]^f \ar[d]^{\phi}		& \cX' \ar[d]^{\phi'} \\
 Y \ar[r]^g							& Y'
 }$$
 is cartesian.
 \end{prop}
 
 \begin{proof}
 By applying Zariski's Main Theorem, there is a factorization $f: \cX \stackrel{i}{\arr} \tilde \cX \stackrel{\tilde f}{\arr} \cX'$ with $i$ an open immersion and $\tilde f$ finite.
 
 Since $f$ is weakly saturated, it follows that $\cX \subseteq \tilde \cX$ is a saturated open substack.  Therefore, there exists a good moduli space $\phi: \cX \arr Y$ inducing a commutative diagram
 $$\xymatrix{
\cX \ar[r]^f \ar[d]^{\phi}		& \cX' \ar[d]^{\phi'} \\
 Y \ar[r]^g							& Y'
 }$$
with $g$ locally of finite type.  Since $f$ is \'etale, pointwise stabilizer preserving and weakly saturated, it follows from Proposition \ref{square_prop} that $g$ is \'etale and that the diagram is cartesian.
 \end{proof}

 \subsection*{\'Etale local existence}
 Theorem \ref{etale_existence_thm} allows us to deduce the existence of a good moduli space for $\cX$ \'etale locally on $\cX$:

\medskip \noindent
{\it Proof of Theorem $\ref{etale_existence_thm}$. }
Let $\cX_2 = \cX_1 \times_{\cX} \cX_1$ with projections $p_1$ and $p_2$.  By Proposition \ref{eff_descent_prop} applied to one of the projections, there exists a good moduli space $\cX_2 \arr Y_2$.  The two projections $p_1, p_2$ induce two morphisms $q_1, q_2: Y_2 \arr Y_1$ such that $q_i \circ \phi_2 = \phi_1 \circ p_i$ for $i=1,2$.   By \cite[Theorem 4.15(xi)]{alper_good_arxiv}, both $Y_2$ and $Y_1$ are finite type over $S$ and by Proposition \ref{etale_preserving_cor}, $q_1$ and $q_2$ are \'etale.  The induced morphisms $\cX_2 \arr Y_2 \times_{q_i, Y_1, \phi_1} \cX_1$ are isomorphisms by Proposition \ref{square_prop}.  Similarly, by setting $\cX_3 = \cX_1 \times_{\cX} \cX_1 \times_{\cX} \cX_1$, there is a good moduli space $\phi_3: \cX_3 \arr Y_3$.  The \'etale projections $p_{12}, p_{13}, p_{23}: \cX_3 \arr \cX_2$ induce \'etale morphism $q_{12}, q_{13}, q_{23}: Y_3 \arr Y_2$.   In summary, there is a diagram
$$\xymatrix{
\cX_3  \ar@<1ex>[r] \ar@<-1ex>[r] \ar[r] \ar[d]	& \cX_2 \ar@<.5ex>[r]^{p_1} \ar@<-.5ex>[r]_{p_2} \ar[d]	& \cX_1 \ar[r]^f \ar[d]	& \cX \\
Y_3  \ar@<1ex>[r] \ar@<-1ex>[r] \ar[r]	& Y_2 \ar@<.5ex>[r]^{q_1} \ar@<-.5ex>[r]_{q_2}	& Y_1 ,
}$$
where all horizontal arrows are \'etale and the squares $\phi_2 \circ p_{ij} = q_{ij} \circ \phi_3 $ and $\phi_1 \circ p_i = q_i \circ \phi_2 $ are cartesian.

There is an identity map $e: \cX_1 \arr \cX_2$, an inverse map $i: \cX_2 \arr \cX_2$ and a multiplication $m: \cX_2 \times_{p_1, \cX_1, p_2} \cX_2 \cong \cX_3 \stackrel{p_{13}}{\arr} \cX_2$ inducing 2-diagrams:  $p_2 \circ e \iso \id \iso p_1 \circ e$, $i \circ i \iso \id$, $t \circ i = s$, $m \circ (i, \id) \iso e \circ p_1$, $m \circ (\id, i) \iso e \circ p_2$, $(e \circ p_1, \id) \circ m \iso \id \iso (e \circ p_2, \id) \circ m$ and $(m, \id) \circ m \iso (\id, m) \circ m$.  

By universality of good moduli spaces, there is an induced identity map $Y_1 \arr Y_2$, an inverse $Y_2 \arr Y_2$ and multiplication $Y_2 \times_{q_1, Y_1, q_2} Y_2 \arr Y_2$ inducing commutative diagrams (as above) giving $Y_2 \rrarrows Y_1$ an \'etale $S$-groupoid structure. 

We claim that $\Delta: Y_2 \arr Y_1 \times Y_1$ is a monomorphism.  Since it is clearly unramified, it suffices to check that $\Delta$ is geometrically injective.  We may assume $S = \Spec k$ with $k$ algebraically closed.  Let $y_1: \Spec k \arr Y_1$, $x_1: \Spec k \arr \cX_1$ be the unique point in $\phi_1^{-1}(y_1)$ closed in $|\cX_1|$, and $x: \Spec k \arr \cX$ be the image of $x_1$.  Since the square
$$\xymatrix{
BG_x \ar[r] \ar[d]				& BG_x \times_k BG_x \ar[d] \\
\cX_2 \ar[r] \ar[d]			& \cX_1 \times \cX_1 \ar[d] \\
\cX \ar[r]					& \cX \times_k \cX
}$$
is 2-cartesian, it follows that there can be only one preimage of $(y_1, y_1)$ under $\Delta$ and is geometrically injective. 

Therefore, there exist an algebraic space quotient $Y$ and induced maps $\phi: \cX \arr Y$ and $Y_1 \arr Y$.  Consider the diagram
$$\xymatrix{
\cX_2 \ar[r] \ar[d]		& \cX_1 \ar[d] \\
\cX_1 \ar[r] \ar[d]		& \cX \ar[d] \\
Y_1 \ar[r]				& Y.
}$$
Since $\cX_2 \cong \cX_1 \times_{Y_1} Y_2$ and $Y_2 \cong Y_1 \times_Y Y_1$, the top and outer squares above are 2-cartesian.  Since $\cX_1 \arr \cX$ is \'etale and surjective, it follows that the bottom square is cartesian.  By descent, $\phi: \cX \arr Y$ is a good moduli space.   \epf

\begin{remark}  \label{stab_hyp_remark}
The above hypotheses of Theorem \ref{etale_existence_thm} can not be weakened to only require that $f$ is stabilizer preserving at $\xi_1$.  Indeed, in Example \ref{example}, the natural \'etale presentation $f: X \arr \cX$ is stabilizer preserving at the origin and both projections $\ZZ_2 \times X \cong X \times_{\cX} X \rrarrows X$ are weakly saturated.  Clearly $X$ admits a good moduli space since it is a scheme but $\cX$ does not admit a good moduli space.
\end{remark}

As an application of Theorem \ref{etale_existence_thm}, we get the following:

\begin{cor} \label{red_cor} Suppose $\cX$ is an Artin stack locally of finite type over an excellent base scheme $S$.  Then $\cX$ admits a good moduli space if and only if $\cX_{\red}$ does.
\end{cor}

\begin{proof}  If $\phi: \cX \arr Y$ is a good moduli space, then \cite[Lemma 4.14]{alper_good_arxiv} implies that $\cX_{\red} \arr Y_{\red}$ is a good moduli space.

Conversely, suppose $\cX_{\red} \arr Y_1$ is a good moduli space with $Y_1$ an algebraic space.  The question is Zariski-local on $S$ and $Y_1$ since determining whether good moduli spaces of a Zariski-open cover glue depends only on the Zariski topology of $|\cX|$ (see \cite[Proposition 7.9]{alper_good_arxiv}).  Therefore, we may assume that $S$ is affine and $Y_1$ is quasi-compact.  If $Y_1$ is affine, then by \cite[Proposition 3.9 (iii)]{alper_good_arxiv} $\cX$ is cohomologically affine.   (The statement is also clear if $Y_1$ is a scheme.)

Let $U_1=\Spec A \arr Y_1$ be an \'etale presentation, $\cU_1 := \cX_{\red} \times_Y U_1 \arr U_1$ the induced good moduli space and $g_1: \cU_1 \arr \cX_{\red}$ be the projection.  There exists an Artin stack $\cU$ and a surjective \'etale morphism $g: \cU \arr \cX$ such that $g_{\red} = g_1$.  There exists a good moduli space $\cU \arr Y$ yielding a 2-commutative diagram
$$\xymatrix{
								& \cX_{\red} \ar@{^(->}[rr] \ar[dd]	&			& \cX  \\
\cU_1 \ar[ru]^{g_1} \ar@{^(->}[rr]	\ar[dd]	&						& \cU \ar[ur]^g \ar[dd] \\
								& Y_1 \\
U_1 \ar[ru] \ar@{^(->}[rr]						&							& Y.
}$$
Since $g_1$ is the pullback of a morphism of algebraic spaces, it is pointwise stabilizer preserving and universally weakly saturated.  Since both of these properties don't depend on the non-reduced structure, it follows that $g_1$ is pointwise stabilizer preserving and universally weakly saturated.  By applying Theorem \ref{etale_existence_thm}, we conclude that $\cX$ admits a good moduli space. 
\end{proof}

\bibliography{../../references}{}
\bibliographystyle{abbrv}



\end{document}